\documentclass[12pt]{amsart}

\usepackage{amsmath}
\usepackage{amsfonts}

\newtheorem{thm}{Theorem}[section]

\newtheorem{prop}[thm]{Proposition}

\theoremstyle{definition}
\newtheorem{defn}[thm]{Definition}

\theoremstyle{remark}
\newtheorem{rem}[thm]{Remark}
\numberwithin{equation}{section}

\newcommand{\D}{\mathbb{D}}
\newcommand{\R}{\mathbb{R}}
\newcommand{\h}{\mathbb{H}}
\def\Re{{\sf Re}\,}
\def\Im{{\sf Im}\,}

\begin{document}

\title{A non-autonomous version of the Denjoy-Wolff Theorem}

\author{Tiziano Casavecchia}
\address{Dipartimento di Matematica, Largo Bruno Pontecorvo 5, Universita di Pisa, 56127 Pisa, Italia}
\email{t.casavecchia@gmail.com}

\author{Santiago D\'{\i}az-Madrigal}
\address{Departamento de Matem\'atica Aplicada II, Camino de los Descubrimientos (ETSII),
Universidad de Sevilla, 41092 Sevilla, Spain}
\email{madrigal@us.es}



\date{\today}

\thanks{Partially supported by the \textit{Ministerio
de Ciencia e Innovaci\'on} and the European Union (FEDER),
project MTM2006-14449-C02-01, by \textit{La Consejer\'{\i}a de
Educaci\'{o}n y Ciencia de la Junta de Andaluc\'{\i}a} and by
the European Science Foundation Research Networking Programme
HCAA}


\begin{abstract}
The aim of this work is to establish the celebrated Denjoy-Wolff Theorem in the context of generalized Loewner chains. In contrast with the classical situation where essentially convergence to a certain point in the closed unit disk is the unique possibility, several new dynamical phenomena appear in this framework. Indeed, $\omega$-limits formed by suitable closed arcs of circumferences appear now as natural possibilities of asymptotic dynamical behavior.
\end{abstract}

\maketitle

\section{Introduction}

One of the most celebrated results about the dynamics of an arbitrary analytic self-map $\varphi$ of the unit disk is the Denjoy-Wolff Theorem which asserts that if $\varphi$ is not an elliptic automorphism, then the sequence of iterates $(\varphi_n)$ converges to a constant (called the Denjoy-Wolff point of $\varphi$) in the closed unit disk, locally uniformly in $\D$. This theorem has been generalized in many directions beginning with the well-known Classification Theorem which describes the dynamics of a holomorphic self-map of a hyperbolic Riemann surface \cite[Chapter 5]{Milnor}. Extensions to smooth multiple connected domains, to several complex variables and even to complex Banach spaces have been also considered.

In a different line of research, a continuous analog of the Denjoy-Wolff Theorem has been proved for semigroups of analytic self-maps of $\D$ (see remark \ref{rem-semigroups}). Namely, if $(\phi_t)$ is such a semigroup and there is no $\phi_t\ (t>0)$ which is an elliptic automorphism, then again the family of iterates $(\phi_t)$ converges to a constant in the closed unit disk when $t$ tends to $\infty$, locally uniformly in $\D$. Moreover, this constant is a fixed point (in the angular sense if it belongs to $\partial\D$) for all the iterates $\phi_t$ of the semigroup.

Those (uniparametric) semigroups of analytic functions sharing a fixed point in $\D$ can be seen as very particular cases of what are usually called classical (biparametric) evolution families in the unit disk (see section \ref{evolution families} for further details). These families are, basically, the solutions of the famous Radial Loewner Equation which, we want to recall, is the non-autonomous differential equation
$$
\begin{cases}
\dfrac{dz}{dt} =-zp(z,t)& \quad\text{for almost every }t\in\lbrack s,\infty)\\
z(s)   =z&
\end{cases}
$$
where $s\geq0$ and $p:\mathbb{D}\times\lbrack0,+\infty)\rightarrow\mathbb{C}$
is measurable in $t$, holomorphic in $z$, and $p(0,t)=1$ and $\Re p(\cdot,t)\geq 0$, for all
$t\geq 0$.
As we have remarked, the corresponding solutions (which are defined for all $t\geq s$)
$
t \mapsto\varphi_{s,t}(z), \ z\in\D,\ s\geq~0
$
are called in the literature evolution families (sometimes transition functions, semigroup elements, ...). The natural question of asymptotic dynamical behavior of those families has also been treated in the literature sometimes linked to the so-called Loewner parametrization method for univalent functions. Namely, it is known that, for every $s\geq 0$, the sequence of iterates $(\varphi_{s,t})_t$ converges to zero when $t$ goes to $\infty$, locally uniformly in $\D$. For a detailed account of this result and, in general, for the classical Loewner theory, we refer the reader to the excellent monograph \cite[Chapter 6]{Pommerenke}. It is worth mentioning that when the above function $p$ does not depend on $t$ (so we have an autonomous differential equation) and fixing $s=0$, the corresponding solutions generate in a natural way a semigroup of analytic functions whose iterates fix the point zero.

Classical Loewner theory has been also extended in many directions. Among them, we owe to cite the theory of the Chordal Loewner Equation \cite[chapter~IV\S7]{Aleksandrov}, the study of Loewner chains in several complex variables \cite{Gra-Ko-book} or the
Schramm\,--\,Loewner equation (SLE, also known as
stochastic Loewner evolution), introduced in 2000 by
Schramm~\cite{Schramm}. Roughly speaking, SLE is a probabilistic version of the previously known radial and chordal Loewner theory.

Quite recently, a full generalization (up to hyperbolic complex manifolds and without common fixed point requirements) of both the Radial and the Chordal Loewner Equation as well as the theory of semigroups of analytic functions in the unit disk has been obtained \cite{BCM1}, \cite{BCM2}. The corresponding and associated theory of generalized Loewner chains have been exposed in \cite{Contreras-Diaz-Pavel}. In this wide context, we deal with non-autonomous differential equations (we call them Herglotz vector fields by obvious reasons) of the form
$$
\begin{cases}
\dfrac{dz}{dt}=(z-\tau(t))(\overline{\tau(t)}z-1)p(z,t)& \quad\text{for almost every }t\in\lbrack s,\infty)\\
z(s)   =z&
\end{cases}
$$
where $s\geq0$, $\tau:[0,+\infty)\rightarrow\overline{\D}$ is measurable and $p:\mathbb{D}\times\lbrack0,+\infty)\rightarrow\mathbb{C}$
is a generalized Herglotz function in $\D$, that is, $p$
is measurable in $t$, holomorphic in $z$, and, $\Re p(\cdot,t)\geq 0$, for all
$t\geq 0$. Again, it is proved the corresponding solutions $t \mapsto\varphi_{s,t}(z), \ z\in\D,\ s\geq 0$ are defined for all $t\geq s$.

The aim of this paper is that, assuming that the function $\tau$ is constant, to analyze the dynamical behavior of the family $(\varphi_{s,t}(z))_t$, for any $s\geq0$ and any $z\in\D$. We notice that this problem includes as very particular cases the one treated for the classical Radial Loewner Equation and the one related to the Continuous Denjoy-Wolff Theorem. Anyhow, one of the key and new points now is we allow $\tau$ to be located in the boundary of $\D$.

The plan of the paper is the following. At the end of this initial section, we recall several facts about the hyperbolic metric in the unit disk, fixing some notations used throughout the paper. For further details, see \cite[Section 2.2]{Milnor}. In section two, and in order the paper to be reasonable self-contained, we outline briefly some ingredients (evolution families and Herglotz vector fields) of that mentioned generalized Loewner chains theory. In section three, we present the main theorem of the paper which, for the sake of readability, has been divided in two parts: the boundary case (Theorem \ref{mainth boundary}) and the inner case (Theorem \ref{mainth inner}). Moreover, the inner case is shown to be a corollary of the boundary case. In this way, we give an alternative and very different approach to the formerly mentioned asymptotic behavior of solutions of the classical Radial Loewner Equation. Finally, in the last section, we have collected several additional results about dynamical topics frequently considered in the literature about the classical Denjoy-Wolff Theorem.


\subsection{Hyperbolic Geometry in the Poincare Disk}

Given an absolutely continuous curve $\gamma :[a,b]\rightarrow \mathbb{D}$,
the associated hyperbolic length of $\gamma $ is
\begin{equation*}
\ell _{\mathbb{D}}(\gamma ):=2\int_{a}^{b}\frac{|\gamma ^{\prime }(t)|}{%
1-|\gamma (t)|^{2}}dt\in \lbrack 0,+\infty ).
\end{equation*}%
If there is no possible confussion, we will write $\ell _{\mathbb{D}%
}(\gamma ([a,b]))$ instead of $\ell _{\mathbb{D}}(\gamma )$. Additionally,
given two points $z_{1},z_{2}\in \mathbb{D},$ the hyperbolic distance from $%
z_{1}$ to $z_{2}$ is defined as%
\begin{equation*}
\rho _{\mathbb{D}}(z_{1},z_{2})=\inf_{\gamma }\ell _{\mathbb{D}}(\gamma )
\end{equation*}%
where $\gamma $ runs through all absolutely continuous curves $\gamma
:[a,b]\rightarrow \mathbb{D}$ such that $\gamma (a)=z_{1}$ and $\gamma
(b)=z_{2}.$ Using univalent functions, the above
hyperbolic metric and related notions can be also considered in any proper
simply connected domain of the complex plane. Namely, the hyperbolic metric
for the right half-plane $\mathbb{H}$ is
\begin{equation*}
\rho _{\mathbb{H}}(w_{1},w_{2}):=\rho _{\mathbb{D}}(\sigma
^{-1}(w_{1}),\sigma ^{-1}(w_{2})),\text{ }w_{1},w_{2}\in \mathbb{H},
\end{equation*}%
where $\sigma (z):=\dfrac{1+z}{1-z},$ $z\in \mathbb{D}$ is the classical
Cayley map.

We define the (open)
hyperbolic disk of center $a\in \mathbb{D}$ and radius $r>0$ as%
\begin{equation*}
D_{H}(a,r):=\{z\in \mathbb{D}:\rho _{\mathbb{D}}(z,a)<r\}.
\end{equation*}%
This subset $D_{H}(a,r)$ is in fact a circle, that is, an open euclidian
disk $D(\widetilde{a},\widetilde{r})$ where, again, $\widetilde{a}=%
\widetilde{a}(a,r)\in \mathbb{D}$ denotes the center and $\widetilde{r}=%
\widetilde{r}(a,r)>0$ the corresponding radius. There are also other
circles which play a prominent role in the hyperbolic geometry: the
horocycles. Given a point $\xi \in \partial \mathbb{D}$ and a positive
number $k>0,$ the horocycle lied at the point $\xi $ and with factor $k$ is
defined as
\begin{equation*}
\mathrm{Hor}(\xi ,k):=\{z\in \mathbb{D}:k_{\mathbb{D}}(\xi ,z):=\frac{|\xi
-z|^{2}}{1-|z|^{2}}<k\}
\end{equation*}%
and describes an euclidian disk internally tangent to $\partial \mathbb{D}
$ at the point $\xi .$ Roughly speaking, horocycles can be considered as
limits of appropriated hyperbolic disks.

As usual, a closed arc of an euclidian circumference $C=C(p,R)$ of center $%
p\in \mathbb{C}$ and radius $R>0$ is a set of the form
\begin{equation*}
A:=\{p+Re^{i\theta }:\theta _{1}\leq \theta \leq \theta _{2}\},\text{ where }%
0<\theta _{2}-\theta _{1}\leq 2\pi .
\end{equation*}%
In case $\theta _{2}-\theta _{1}<2\pi ,$ the arc $A$ is said to be a proper
arc of the circumference $C$ and the points $p+Re^{i\theta _{1}},$ $%
p+Re^{i\theta _{2}}$ are called the extreme points of $A$. Circumferences
are also present in the hyperbolic context. Namely, the hyperbolic circumference of center $a\in \mathbb{D}$ and radius
$r>0$ is defined as $\ C_{H}(a,r):=\{z\in \mathbb{D}:\rho _{\mathbb{D}%
}(z,a)=r\}.$ Certainly, $C_{H}(a,r)=C(\widetilde{a}(a,r),\widetilde{r}%
(a,r)), $ for some $\widetilde{a}\in \mathbb{D}$ and $\widetilde{r}>0,$.

We recall that the hyperbolic angular extent of a closed arc $A$ of $%
C_{H}(a,r)$ is the number $\dfrac{\ell _{H}(A)}{\sinh (r)}.$ This definition
is modeled on the well-known formula saying that the hyperbolic length of $%
C_{H}(0,r)$ is $2\pi \sinh (r).$ Moreover, if $\gamma (\theta )=Re^{i\theta
},$ with $R\in (0,1)$ and $\theta \in \lbrack \theta _{1},\theta _{2}]$ $%
(0<\theta _{2}-\theta _{1}\leq 2\pi ),$ it can be checked that
\begin{equation*}
\dfrac{\ell _{H}(\gamma )}{\sinh (r)}=\theta _{2}-\theta _{1},
\end{equation*}%
where $r=\log \frac{1+R}{1-R}.$ In fact, $C(0,R)=C_{H}(0,r).$ In other
words, when the center is the origin, the hyperbolic angular extent of a
closed arc is nothing but its usual euclidian angular extent.

Bearing in mind the above comments and the results of this paper, the following definition seems to
be appropriated.

\begin{defn}
Given $\xi \in \partial \mathbb{D}$ and $k>0$, the hyperbolic angular extent
of a closed arc $A\subset \mathbb{D}$ located on the boundary of the
horocycle $\mathrm{Hor}(\xi ,k)$ will be defined as $\dfrac{\ell _{H}(A)}{k}\in
[0,+\infty ).$
\end{defn}


\section{Loewner Chains, Evolution Families and Herglotz Vector Fields}\label{evolution families}
All the material shown in this section is contained in \cite{BCM1} or \cite{Contreras-Diaz-Pavel} or can be easily deduced from results appearing in these two papers.

\begin{defn}
A biparametric ($0\leq s\leq t<+\infty$) family $(\varphi_{s,t})$ of holomorphic self-maps
of the unit disk is called a (continuous) evolution family if

\begin{enumerate}
\item $\varphi_{s,s}=id_{\mathbb{D}},$ for all $s\geq0,$

\item $\varphi_{s,t}=\varphi_{u,t}\circ\varphi_{s,u}$ for all $0\leq
s\leq u\leq t<+\infty,$

\item for all $z\in\mathbb{D}$ and for all $T>0$ there exists a
non-negative function $k_{z,T}\in L^{1}([0,T],\mathbb{R})$ such that
\begin{equation*}
|\varphi_{s,u}(z)-\varphi_{s,t}(z)|\leq\int_{u}^{t}k_{z,T}(\xi)d\xi
\end{equation*}
for all $0\leq s\leq u\leq t\leq T.$
\end{enumerate}
\end{defn}

\begin{rem}\label{rem-semigroups} Perhaps, the simplest example of those families is the one related to semigroups of analytic functions $(\phi_t)$. This notion means that $t\mapsto\phi_{t}$ is a continuous homomorphism
from the additive semigroup of non-negative real numbers into
the composition semigroup of holomorphic self-maps of
$\D$. Now, defining $\varphi_{s,t}:=\phi_{t-s}$, for  $0\leq s\leq
t<+\infty$, we have $(\varphi_{s,t})$ is an evolution family in $\D$.
For further information about semigroups of analytic functions, we refer the reader to
\cite{R-S2} or \cite{Shoiket}.
\end{rem}
\begin{rem}
Evolution families are always univalent and locally absolutely
continuous with respect to the variables $s$ and $t.$
\end{rem}

With the help of appropriated univalent maps, it is possible to transfer the
above definition to general simply connected domains different from the whole complex plane. For instance, a
biparametric family $(\phi _{s,t})$ of holomorphic self-maps of $\mathbb{H}$ is said to
be a (continuous) evolution family in $\mathbb{H}$, if for some
biholomorphic function $f$ mapping $\mathbb{D}$ onto $\mathbb{H}$, the
biparametric family%
\begin{equation*}
\varphi _{s,t}(z):=f^{-1}\circ\phi _{s,t}\circ f(z),\text{ }%
z\in \mathbb{D}\text{, }0\leq s\leq t,
\end{equation*}%
is a (continuous) evolution family in $\mathbb{D}$. It can be checked that,
here, the word ``some" can be replaced by ``any".

A fundamental fact about evolution families is that they are solutions (in the weak sense) of certain non-autonomous differential equations. Following Cara\-theodory's ODE theory, we recall the notion of weak holomorphic vector field which allows to assure the existence of well-defined ``trajectories'' for the corresponding initial value problems associated with arbitrary initial data in $[0,+\infty)\times\D$.

\begin{defn}
\label{Definicion-VF} A weak holomorphic vector field on a simply connected domain $\Omega$
is a function
$G:\Omega\times\lbrack0,+\infty)\rightarrow \mathbb{C}$
with the following properties:

\begin{enumerate}
\item For all $z\in\Omega,$ the function $t\in\lbrack
0,+\infty)\mapsto G(z,t)$ is measurable.

\item For all $t\in\lbrack0,+\infty),$ the function $
z\in \Omega\mapsto G(z,t)$ is holomorphic.

\item For any compact set $K\subset\Omega$ and for all $T>0$ there
exists a non-negative function $k_{K,T}\in
L^{1}([0,T],\mathbb{R})$ such that
\[
|G(z,t)|\leq k_{K,T}(t)
\]
for all $z\in K$ and for almost every $t\in\lbrack0,T].$
\end{enumerate}
\end{defn}

The relationship between evolution families and weak holomorphic vector fields is presented in our next two theorems (see also Theorems 4.8, 5.2 and 6.2 from \cite{BCM1}).

\begin{thm}
Let $(\varphi _{s,t})$ be an evolution family in the unit disk. Then, there
exists a weak holomorphic vector field $G$ in the unit disk such that, for
any $z\in \mathbb{D}$ and for any $s\geq 0$, the solution (in the weak
sense) of the Cauchy problem
\begin{equation*}
\left\{
\begin{array}{ll}
\dfrac{dz}{dt}=G(z,t) & t\geq s \\
z(s)=z &
\end{array}%
\right.
\end{equation*}%
is exactly the map $t\in \lbrack s,+\infty )\mapsto \varphi _{s,t}(z)\in
\mathbb{D}$. Moreover, there exist $%
\tau :[0,+\infty )\mathbb{\longrightarrow }\overline{\mathbb{D}}$ measurable
and $p:\mathbb{D}\times \lbrack 0,+\infty )\longrightarrow \mathbb{C}$ a
generalized Herglotz function in the unit disk such that%
\begin{equation*}
G(z,t)=(z-\tau (t))(\overline{\tau (t)}z-1)p(z,t)\text{  }z\in
\mathbb{D},\text{ }t\geq 0.
\end{equation*}
\end{thm}
In what follows, we will refer to any such
weak holomorphic vector field as an (the) associated vector field of the evolution family, since all of them are essentially the same (they are equal almost everywhere in $t$).

\begin{thm}
Given $\tau :[0,+\infty )\mathbb{\longrightarrow }\overline{\mathbb{D}}$ measurable
and $p:\mathbb{D}\times \lbrack 0,+\infty )\longrightarrow \mathbb{C}$
a generalized Herglotz function in $\mathbb{D}$, the solutions $\psi \lbrack
z,s]$ of the (weak) Cauchy problems $(z\in \mathbb{D}$ and $s\geq 0)$
\begin{equation*}
\left\{
\begin{array}{ll}
\dfrac{dz}{dt}=(z-\tau (t))(\overline{\tau (t)}z-1)p(z,t) & t\geq s \\
z(s)=z &
\end{array}%
\right.
\end{equation*}%
are always defined in the whole interval $[s,+\infty )$ and the biparametric
family
\begin{equation*}
\varphi _{s,t}(z):=\psi \lbrack z,s](t),\text{ }z\in \mathbb{D},\text{ }%
t\geq s
\end{equation*}%
is an evolution family in the unit disk.
\end{thm}

It is really straightforward to check (really to compute) that the above
two theorems still hold changing $\D$ by $\h$. In fact, the corresponding weak holomorphic vector fields associated with evolution families in the
right-half plane $(\phi _{s,t})$ are exactly the following ones:
\begin{equation*}
F(w,t):=f^{\prime }(f^{-1}(w))G(f^{-1}(w),t),\text{ }w\in
\mathbb{H},\text{ }t\geq 0,
\end{equation*}%
where $f$ is any biholomorphic function mapping $\mathbb{D}$ onto $\mathbb{H}$ and $G$ is a/the associated vector field with the evolution family in $\mathbb{D}$ defined as $\varphi _{s,t}:=f^{-1} \circ \phi _{s,t}\circ f$.

In this context, it is possible to define (in a intrinsic way) the notion of Loewner chains and to show an essentially one-to-one correspondence among this concept and the ones mentioned above of evolution families and Herglotz vector fields. Certainly, this new notion includes and extends the classical one (see \cite{Contreras-Diaz-Pavel} for further details).

\subsection{Evolution Families with Common Denjoy-Wolff Point}\label{sub DW-point}

\begin{defn}
An evolution family $(\varphi _{s,t})$ in $\mathbb{D}$ is said to have a
common Denjoy-Wolff point $\tau \in \overline{\mathbb{D}}$, if the
Denjoy-Wolff point of each $\varphi _{s,t},$ other than the identity, is $\tau
.$
\end{defn}

\begin{rem}
We note that the trivial evolution family\ (every $\varphi _{s,t}$ is the
identity) has common Denjoy-Wolff point $\tau ,$ for each $\tau \in
\overline{\mathbb{D}}$. Appart from this extreme case, the above point $\tau
$ is univocally determined.
\end{rem}

\begin{rem} By the inner case, we mean that the above point $\tau$ belongs to $\D$. Likewise, the boundary case indicates that $\tau$ belongs to $\partial\D$. For several reasons (angular concepts, proof techniques, ...), it is convenient to handle these two cases separately.
\end{rem}
\begin{rem} Evolution families associated with semigroups of analytic functions are, all of them, evolution families with a common Denjoy-Wolff point $\tau\in\overline{\D}$. Also the evolution families related to the Radial Loewner Equation are evolution families with common Denjoy-Wolff point, but now $\tau\in \D$ (usually $\tau=0$). A similar fact happens with the Chordal Loewner Equation and this time $\tau\in\partial\D$.
\end{rem}

In a similar way as in discrete iteration theory in the unit disk, the boundary case, that is, evolution
families $(\varphi_{s,t})$ having a common Denjoy-Wolff point $\tau \in \partial \mathbb{D}$
are better studied by means of the associated evolution families in the
right-half plane $\phi _{s,t}:=\sigma _{\tau }\circ \varphi _{s,t}\circ
\sigma _{\tau }^{-1}$ $(0\leq s\leq t),$ where $\sigma _{\tau }(z):=\dfrac{%
\tau +z}{\tau -z},$ $z\in \mathbb{D}$. That is, $\sigma _{\tau }$ is the
classical Cayley map associated with $\tau .$ We recall that $\sigma _{\tau
} $ is a biholomorphic map from $\mathbb{D}$ onto $\mathbb{H}$ which admits
a bicontinuous extension from $\overline{\mathbb{D}}$ onto $cl_{\infty }%
\mathbb{H}$, the closure of $\mathbb{H}$ in the Riemann sphere $\mathbb{C}%
_{\infty }.$

Now, the relationship between evolution families with a common Denjoy-Wolff point and weak holomorphic vector fields is as follows.

\begin{prop}\label{BPorta formula with DW-common}
Let $(\varphi _{s,t})$ a non-trivial evolution family in the unit disk with common
Denjoy-Wolff point $\tau \in \overline{\mathbb{D}}$. Then, there exists a generalized Herglotz function in the unit disk $p$ such that
\begin{equation*}
G(z,t)=(z-\tau )(\overline{\tau }z-1)p(z,t),\text{ }z\in \mathbb{D},\text{ }%
t\geq 0
\end{equation*}%
is an associated vector field of the family $(\varphi _{s,t}).$

Reciprocally, given $\tau\in\overline{\D}$ and $p$ a generalized Herglotz function in the unit disk, the corresponding function $G$ defined as above is the associated vector field of some evolution family in the unit disk with common
Denjoy-Wolff point $\tau \in \overline{\mathbb{D}}$.

\end{prop}

In the boundary case and in the right-half plane context, the above theorem takes a more simplified form. We note a generalized Herglotz function $P$ in the right half-plane simply means that $P\circ\sigma^{-1}$ is a generalized Herglotz function in $\D$.

\begin{prop}\label{Herglotz in H}
Let $(\phi _{s,t})$ be
a non-trivial evolution family in the right-half plane having $\infty$ as common Denjoy-Wolff point. Then, there exists a generalized Herglotz function $P$ in the right-half plane such that $P$ is an associated vector field of the
family $(\phi _{s,t}).$

Reciprocally, any generalized Herglotz function $P$ in $\h$ is the the associated vector field of some evolution family in the right-half plane with common
Denjoy-Wolff point $\infty$.
\end{prop}

\section{Main Results}

We are now ready to give the main results of this paper. Bearing in mind the particularities of each case, and not only for clarity, we have divided our version of the Denjoy-Wolff Theorem in two parts: the inner case and the boundary case. These notions refer to the material exposed in subsection (\ref{sub DW-point}). We begin by dealing with the boundary case.

\begin{thm}\label{mainth boundary}
Let $(\varphi _{s,t})$ be a non-trivial evolution family in the unit disk with common
Denjoy-Wolff point $\tau \in \partial \mathbb{D}$. Then, one and only one of
the three mutually excluding situations can happen:

\begin{enumerate}
\item For every $s\geq 0,$
$$
\lim_{t\rightarrow +\infty }\varphi
_{s,t}=\tau
$$
 uniformly on compact subsets of $\mathbb{D}.$

\item For every $s\geq 0$, there exists a univalent self-map of the unit
disk $h_{s}$ such that
$$
\lim_{t\rightarrow +\infty }\varphi _{s,t}=h_{s}
$$ uniformly on compact subsets of $\mathbb{D}.$

\item For every $s\geq 0$ and for every $z\in \mathbb{D},$ the $\omega $%
-limit $\omega (s,z)$ of the trajectory%
\begin{equation*}
t\in \lbrack s,+\infty )\longmapsto \varphi _{s,t}(z)\in \mathbb{D}
\end{equation*}%
is a closed arc of the circumference defined by the boundary of a certain horocycle
$
\mathrm{Hor}(\mathbb{\tau },k(s,z))$ where
$0<k(s,z)\leq k_{\mathbb{D}}(\tau,z).
$
\newline
Moreover, this case holds if and only if one of the following three
mutually excluding subcases holds:

\begin{enumerate}
\item For every $s\geq 0$ and for every $z\in \mathbb{D},$ $\omega (s,z)$ is
exactly the whole circumference $\partial \mathrm{Hor}(\mathbb{\tau },k(s,z)).$

\item For every $s\geq 0$ and for every $z\in \mathbb{D},$
$\omega (s,z)$ is a proper closed arc of $\partial \mathrm{Hor}(\mathbb{\tau},k(s,z))$
and one of its extreme points is $\tau .$

\item For every $s\geq 0$ and for every $z\in \mathbb{D}$,
$\omega (s,z)$ is a proper closed arc of
$\partial \mathrm{Hor}(\mathbb{\tau},k(s,z))$, it is contained in $\D$, and all of those arcs have the same
associated hyperbolic angular extent.
\end{enumerate}
\end{enumerate}
\end{thm}

\begin{proof}
First of all, we move to the right-half plane context. That is (see section \ref{evolution families}), we consider $%
(\phi _{s,t})$ the evolution family in the right-half plane associated with
$(\varphi _{s,t}).$ We know that the Denjoy-Wolff point of each $\phi
_{s,t}$ different from the identity is $\infty \in \partial _{\infty }%
\mathbb{H}$ and an associated vector field with $(\phi
_{s,t}) $ is given by some generalized Herglotz function in the right-half
plane $F:\mathbb{H}\times \lbrack 0,+\infty )\rightarrow \mathbb{C}.$

We begin by proving several steps of independent interest and, at the end,
using these steps, we will present properly the proof of the theorem.

\smallskip

\noindent [Step I] The following are equivalent:

$(i)$ For any $w\in \mathbb{H}$ and for any $s\geq 0,$ $\lim_{t\rightarrow
+\infty }\phi _{s,t}(w)=\infty .$

$(ii)$ There exist $w_{0}\in \mathbb{H}$ and $s_{0}\geq 0$ such that $%
\lim_{t\rightarrow +\infty }\phi _{s_{0},t}(w_{0})=\infty .$

Moreover, when item $(i)$ holds, the family $(\phi _{s,t})_t$ converges to $\infty$ uniformly on compact subsets of $\h$, for every $s\geq 0$.

\smallskip

\noindent [Proof of Step I] Assume $(ii)$ and take $w_{1}\in \mathbb{H}$ different
from $w_{0}.$ Consider (and fix) a certain strictly increasing sequence $%
(t_{n})$ of real numbers in $[s_{0},+\infty )$ converging to \thinspace $%
+\infty .$ Since, by Montel's Theorem, $(\phi _{s_{0},t_{n}})_{n}$ is a
normal family, we can find a subsequence $(t_{n_{k}})$ and $p\in cl_{\infty }%
\mathbb{H}$ such that $\lim_{k\rightarrow +\infty }\phi
_{s_{0},t_{n_{k}}}(w_{1})=p.$ Moreover, because of holomorphy, we find that,
for all $k,$
\begin{equation*}
\rho _{\mathbb{H}}^{{}}(\phi _{s_{0},t_{n_{k}}}(w_{1}),\phi
_{s_{0},t_{n_{k}}}(w_{0}))\leq \rho _{\mathbb{H}}^{{}}(w_{1},w_{0}).
\end{equation*}%
By hypothesis, $\lim_{t\rightarrow +\infty }\phi _{s_{0},t}(w_{0})=\infty $
so $p\in \partial _{\infty }\mathbb{H}.$ Assume for a moment that $p\neq
\infty .$ Then,
\begin{eqnarray*}
\Re p &=&\lim_{k\rightarrow +\infty }\Re\phi
_{s_{0},t_{n_{k}}}(w_{1})=\lim_{k\rightarrow +\infty }\Re\left[
w_{1}+\int_{s}^{t_{n_{k}}}F(\phi _{s_{0},\xi }(w_{1}),\xi )d\xi \right] \\
&=&\Re w_{1}+\lim_{k\rightarrow +\infty }\int_{s}^{t_{n_{k}}}\Re%
F(\phi _{s_{0},\xi }(w_{1}),\xi )d\xi \geq \Re w_{1}>0.
\end{eqnarray*}%
Therefore, $p$ is not purely imaginary and we obtain a contradiction. Bearing in
mind that the sequence $(t_{n})$ was chosen arbitrarily, we deduce that $%
\lim_{t\rightarrow +\infty }\phi _{s_{0},t}(w)=\infty ,$ for all $w\in
\mathbb{H}.$

Now, fix $s\geq 0$ different from $s_{0}.$ If $s<s_{0},$ then (considering
$t~$large enough), for any $w\in \mathbb{H},$
\begin{equation*}
\lim_{t\rightarrow +\infty }\phi _{s,t}(w)=\lim_{t\rightarrow +\infty }\phi
_{s_{0},t}(\phi _{s,s_{0}}(w))=\infty .
\end{equation*}%
Likewise, if $s>s_{0},$ then, for any $w\in \mathbb{H},$%
\begin{equation*}
\infty =\lim_{t\rightarrow +\infty }\phi _{s_{0},t}(w)=\lim_{t\rightarrow
+\infty }\phi _{s,t}(\phi _{s_{0},s}(w)).
\end{equation*}%
That is, we have found that $(\phi _{s,t})_{t}$ tends pointwise to $\infty $
on the set $\phi _{s_{0},s}(\mathbb{H})$. Since $\phi _{s_{0},s}$ is
univalent, we have that $\phi _{s_{0},s}(\h)$ is an open set of $\mathbb{H}$ so, by Vitali's
Theorem, we conclude that $\lim_{t\rightarrow +\infty }\dfrac{1}{\phi
_{s,t}(w)}=0,$ for all $w\in \mathbb{H}$ and, therefore, $\lim_{t\rightarrow
+\infty }\phi _{s,t}(w)=\infty ,$ for all $w\in \mathbb{H}.$

Finally, applying again Montel's Theorem, we deduce that
$\lim_{t\rightarrow
+\infty }\phi _{s,t}=\infty $ uniformly on compact subsets of $\h$, for every $s\geq 0$.

\smallskip

\noindent [Step II] The following dichotomy holds:

$(i)$ Either, for every $w\in \mathbb{H}$ and every $s\geq 0,$
\begin{equation*}
\lim_{t\rightarrow +\infty }\Re\phi _{s,t}(w)=+\infty ,
\end{equation*}

$(ii)$ or, for every $s\geq 0,$ there exists a holomorphic map $\theta
_{s}\in \mathrm{Hol}(\mathbb{H};\mathbb{H})$ such that
\begin{equation*}
\phi _{s,t}-i\Im\phi _{s,t}(1)\overset{t\rightarrow +\infty }{%
\longrightarrow }\theta _{s}
\end{equation*}%
uniformly on compact subsets of $\mathbb{H}$. In particular, for every $s\geq 0,$ the map
\begin{equation*}
w\in \mathbb{H}\mapsto \lim_{t\rightarrow +\infty }\Re\phi
_{s,t}(w)\in \mathbb{R}
\end{equation*}%
is a well-defined harmonic function in the right half-plane.

\smallskip

\noindent [Proof of Step II] Assume that $(i)$ fails. This means that there exist $
s_{0}\geq 0$, $w_{0}\in \mathbb{H}$ and a sequence $(t_{n})\subset \lbrack
s_{0},+\infty )$ converging to $+\infty $ such that
\begin{equation*}
\sup \{\Re\phi _{s_{0},t_{n}}(w_{0}):n\in \mathbb{N}\}<+\infty \text{.}
\end{equation*}%
However, for every $s\geq 0$ and for every $w\in \mathbb{H}$, the map $t\in
\lbrack s,+\infty )\rightarrow \Re\phi _{s,t}(w)\in~\mathbb{R}$ is non-decreasing. Indeed, given any $0\leq s\leq u\leq t,$
\begin{eqnarray*}
\Re\phi _{s,u}(w) &=&\Re\left[ w+\int_{s}^{u}F(\phi _{s,\xi
}(w),\xi )d\xi \right] \\
&\leq &\Re w+\int_{s}^{t}\Re F(\phi _{s,\xi }(w),\xi )d\xi =\Re\phi _{s,t}(w).
\end{eqnarray*}%
Therefore, we can conclude that, in fact, $\sup \{\Re\phi
_{s_{0},t}(w_{0}):t\geq s_{0}\}<+\infty $. Now, by the right-half plane
version of Harnack's inequality, we see that, $\text{for every }w\in \mathbb{H}$,
\begin{equation*}
\Re\phi _{s_{0},t}(1)\frac{\Re w}{|w|^{2}+1}\leq \Re\phi
_{s_{0},t}(w)\leq \Re\phi _{s_{0},t}(1)\frac{|w|^{2}+1}{\Re w}.
\end{equation*}%
Hence, $\sup \{\Re\phi _{s_{0},t}(w):t\geq s_{0}\}<+\infty ,$ for
every $w\in \mathbb{H}$. Moreover, bearing in mind that $\phi _{s,t}=\phi _{s_{0},t}\circ
\phi _{s,s_{0}},$ whenever $0\leq s\leq s_{0}\leq t,$ we deduce that,$\text{ for every }w\in
\mathbb{H}\text{ and every }s\in \lbrack 0,s_{0}],$
\begin{equation*}
\sup \{\Re\phi _{s,t}(w):t\geq s\}<+\infty.
\end{equation*}
Consider now some $s\geq s_{0}.$ Arguing as before, we can find $w_{1}\in
\mathbb{H}$ (take, for instance, $w_{1}:=\phi_{s_0,s}(1)$) such that $\sup \{\Re\phi _{s,t}(w_{1}):t\geq s\}<+\infty
. $ Now, applying again Harnack's inequality, we also obtain that
$\sup \{\Re\phi _{s,t}(w):t\geq s\}<+\infty ,$ for every $w\in \mathbb{H%
}$.

These previous results allow us to assert, for every $w\in \mathbb{H}$ and
every $s\geq 0,$ the existence of the limit
\begin{equation*}
(\ast )\text{ }H(w,s):=\lim_{t\rightarrow +\infty }\Re\phi
_{s,t}(w)=\sup\{\Re\phi _{s,t}(w):t\geq s\}\in (0,+\infty ).
\end{equation*}

Now, fix $s\geq 0$ and denote $\theta _{s,t}(w):=\phi _{s,t}(w)-i\Im%
\phi _{s,t}(1),$ where $w\in \mathbb{H}$ and $t\geq s.$ Trivially, $\theta
_{s,t}\in \mathrm{Hol}(\mathbb{H};\mathbb{H})$ so, by Montel's Theorem,
$\{\theta _{s,t}:t\geq s\}$ is a normal family in $\mathbb{H}$. Moreover,
since, for every $t\geq s$,
\begin{equation*}
|\theta _{s,t}(1)|=|\Re\phi _{s,t}(1)|\leq H(1,s)<+\infty,
\end{equation*}
we conclude that $\{\theta _{s,t}:t\geq s\}$ is indeed a relatively
compact subset of $\mathrm{Hol}(\mathbb{H};\mathbb{C}).$ Consider two
arbitrary accumulation points $h_{1}$ and $h_{2}$ of that set when tending $%
t $ to infinite. Of course, this means that $h_{1},h_{2}\in $ $\mathrm{Hol}(%
\mathbb{H};\mathbb{C})$ and there exist two increasing sequences $(t_{n})$
and $(u_{n})$ of real numbers belonging to $[s,+\infty )$ and converging to +%
$\infty $ such that
\begin{equation*}
h_{1}=\lim_{n\rightarrow +\infty }\theta _{s,t_{n}}\text{ and }%
h_{2}=\lim_{n\rightarrow +\infty }\theta _{s,u_{n}}
\end{equation*}
uniformly on compact subsets of $\mathbb{H}$. Bearing in mind $(\ast ),$ we
note that, for every $w\in \mathbb{H},$%
\begin{equation*}
\Re h_{1}(w)=\lim_{t\rightarrow +\infty }\Re\theta
_{s,t_{n}}(w)=H(w,s)=\lim_{t\rightarrow +\infty }\Re\theta
_{s,u_{n}}(w)=\Re h_{2}(w).
\end{equation*}%
In other words, $h_{1}$ and $h_{2}$ are two holomorphic maps with the same
real part. Necessarily, there exists $c\in \mathbb{C}$, such that,
\begin{equation*}
h_{1}(w)=h_{2}(w)+c,\text{ for all }w\in \mathbb{H}\text{.}
\end{equation*}%
However, using again $(\ast ),$%
\begin{eqnarray*}
c &=&h_{1}(1)-h_{2}(1)=\lim_{n\rightarrow +\infty }(\theta
_{s,t_{n}}(1)-\theta _{s,u_{n}}(1))=\lim_{t\rightarrow +\infty }(\Re%
\theta _{s,t_{n}}(1)-\Re\theta _{s,u_{n}}(1)) \\
&=&H(1,s)-H(1,s)=0.
\end{eqnarray*}%
Therefore, $h_{1}=h_{2}.$ Since those accumulation points were chosen
arbitrarily, we can consider
\begin{equation*}
\theta _{s}:=\lim_{t\rightarrow +\infty }(\phi _{s,t}-i\Im\phi
_{s,t}(1))\in \mathrm{Hol}(\mathbb{H};\mathbb{C}).
\end{equation*}
Certainly, every $\phi _{s,t}$ is univalent (see section \ref{evolution families}) so, by Hurwitz`s Theorem,
either $\theta _{s}$ is univalent or is a constant. We also note that $
\Re\theta _{s}\geq 0.$ Now, in case $\theta _{s}$ is univalent,
applying the Open Mapping Theorem, we deduce that $\theta _{s}\in \mathrm{Hol%
}(\mathbb{H};\mathbb{H}).$ On the other hand, if $\theta _{s}=a+ib$ for some
$a+ib\in \mathbb{C},$ and thinking off the Julia-Caratheodory Theorem, we find
that, for each $w\in \mathbb{H}$,
\begin{equation*}
\Re a=\Re\theta _{s}(w)=\lim_{t\rightarrow +\infty}
\Re\phi _{s,t}(w)\geq \Re w.
\end{equation*}%
Therefore, $\Re a\geq \Re w,$ for every $w\in \mathbb{H}$, and we
obtain a contradiction. In other words, this second case is impossible.

\smallskip

\noindent [Step III]: Assume that [Step I fails]. Then, there exists a non-empty subset
$A\subset \mathbb{R}\cup \{\infty \}$ (only depending on the whole evolution
family but neither on $s$ nor on $w)$ such that, for all $w\in \mathbb{H}$ and
for all $s\geq 0,$
\begin{equation*}
\omega (s,w)=f_{s}(w)\oplus iA:=\{f_{s}(w)+ia:a\in A\},
\end{equation*}
where $f_{s}$ is a certain univalent function
belonging to $\mathrm{Hol}(\mathbb{H};\mathbb{H})$ and $\omega (s,w)$ is the
$\omega $-limit of the trajectory
\begin{equation*}
t\in \lbrack s,+\infty )\mapsto \phi _{s,t}(w)\in \mathbb{H}
\end{equation*}
(for any $p\in \mathbb{C}$, the sum $^{\prime \prime}p+i\infty^{\prime \prime}$ must be
understood as $\infty $).\par
Moreover, $\Re f_{s}(w)\geq \Re w,$ for all $w\in \mathbb{H}$ and for all $s\geq 0.$ In addition, $A$ is either a real
number, a compact interval, a proper unbounded closed interval union with $\infty $
or is the whole $\mathbb{R}$ union with $\infty .$

\smallskip

\noindent [Proof of Step III] Fix $s\geq 0.$ Let us analyze briefly $\omega (s,1)$,
the $\omega $-limit of the trajectory
\begin{equation*}
t\in \lbrack s,+\infty )\mapsto \phi _{s,t}(1)\in \mathbb{H}.
\end{equation*}%
Since $\mathbb{C}_{\infty }$ is a compact metric space, it is well-known $%
\omega (s,1)$ is a non-void compact and connected subset of the Riemann
sphere $\mathbb{C}_{\infty }$. Moreover, since [Step I] fails and, by using
[Step II], the existence of the limit $\Re\theta
_{s}(1)=\lim_{t\rightarrow +\infty }\Re\phi _{s,t}(1)\in \mathbb{R}$
is always guaranteed. Therefore, bearing in mind the definition of the
notion of $\omega $-limit, we see that
\begin{equation*}
\omega (s,1)\setminus \{\infty \}=\Re\theta _{s}(1)\oplus iB_{s},
\end{equation*}%
where $B_{s}$ is the $\omega $-limit (in the real line) of the function $t\in \lbrack s,+\infty
)\mapsto \Im\phi _{s,t}(1)\in \mathbb{R}$, that is,
\begin{equation*}
B_{s}:=\{r\in \mathbb{R}:\text{ There exists }(t_{n})\subset \lbrack
s,+\infty )\text{ with }t_{n}\overset{n}{\longrightarrow }+\infty \text{ and
}\Im\phi _{s,t_{n}}(1)\overset{n}{\longrightarrow }r\}.
\end{equation*}%
The set $B_s$ is closed because $\R$ is a metric space. Moreover, by the Value Intermediate Theorem, we obtain that $B_{s}$ must be a point or
a closed interval of the real line. Therefore, $B_{s}$ must be a real number, a
closed finite interval, a closed unbounded interval different from $\R$ or the whole real line.
However, since $\omega (s,1)$ must be closed in the Riemann sphere, we
deduce that
\begin{equation*}
\omega (s,1)=\Re\theta _{s}(1)\oplus iA_{s}.
\end{equation*}%
where $A_{s}\subset \mathbb{R}\cup \{\infty \}$ is either a real number, a
compact interval, a proper unbounded closed interval union with $\infty $ or the
whole $\mathbb{R}$ union with $\infty .$

Now, take an arbitrary $w\in \mathbb{H}$ and consider the following
decomposition (see again [Step II] for the definition of $\theta _{s,t}$):%
\begin{equation*}
\phi _{s,t}(w)=\theta _{s,t}(w)-\Re\phi _{s,t}(1)+\phi _{s,t}(1),\text{
}t\geq s.
\end{equation*}%
Since the first two summands on the right side have limit when $t$ tends to $%
+\infty $, we deduce that
\begin{equation*}
\omega (s,w)=(\theta _{s}(w)-\Re\theta _{s}(1))\oplus \omega (s,1).
\end{equation*}%
From here, it is clear that $\omega (s,w)=\theta _{s}(w)+iA_{s}.$ Now, we
will show $A_{s}=h(s)\oplus A_{0}$ for a certain map $h:[0,+\infty
)\rightarrow \mathbb{R}$ \ (as above, if $p\in \mathbb{R}$, the sum $%
^{\prime \prime }p+\infty ^{\prime \prime }$ must be understood as $\infty ).$ So, fix $%
s\geq 0.$ By the own definition of evolution families, we know that%
\begin{equation*}
\phi _{0,t}(w)=\phi _{s,t}(\phi _{0,s}(w))\,,\text{ for all }w\in \mathbb{H}%
\text{ and for all }t\geq s.
\end{equation*}
From this identity, we immediately obtain that
\begin{equation*}
\theta _{0}(w)\oplus iA_{0}=\omega (0,w)=\omega (s,\phi _{0,s}(w))=\theta
_{s}(\phi _{0,s}(w))\oplus iA_{s}.
\end{equation*}%
In particular, since $\Im\theta _{0}(1)=0,$ we deduce that
\begin{equation*}
A_{0}=\Im\theta _{s}(\phi _{0,s}(1))\oplus A_{s}.
\end{equation*}%
Therefore, the function we are looking for is $h(s):=-\Im\theta
_{s}(\phi _{0,s}(1)),$ $s\geq 0.$ Finally, defining
\begin{equation*}
A:=A_{0}\text{ and }f_{s}(w):=\theta _{s}(w)+i\Im\theta _{s}(\phi
_{0,s}(1)),
\end{equation*}
we see that $\omega (s,w)$ $=f_{s}(w)\oplus iA,$ for all $w\in \mathbb{H}$
and for all $s\geq 0$ as wanted. Looking once more at [Step II], we trivially deduce $f_{s}$ is a univalent function belonging to $\mathrm{Hol}(\mathbb{H};\mathbb{H}).$ Moreover, for every $w\in \mathbb{H},$
\begin{eqnarray*}
\Re f_{s}(w) =\Re\theta _{s}(w)=\lim_{t\rightarrow +\infty }\Re\phi_{s,t}(w)\geq \Re w.
\end{eqnarray*}

\smallskip

\noindent [Step IV] Assume that [Step I] fails and let $f_{s}$ be the function defined
in [Step III]. Then one and only one of the four following cases holds:

$(IV.1)$ For all $s\geq 0,$ there exists a univalent function $h_{s}\in
\mathrm{Hol}(\mathbb{H};\mathbb{H})$ such that $\lim_{t\rightarrow +\infty
}\phi _{s,t}=h_{s}$ uniformly on compact subsets of $\mathbb{H}$.

$(IV.2)$ For all $s\geq 0$ and for all $w\in \mathbb{H},$ we have that
$$
\omega (s,w)=\{\infty \}\cup \{\Re f_{s}(w)+ia:a\in \mathbb{R}\}.
$$

$(IV.3)$ For all $s\geq 0$ and for all $w\in \mathbb{H}$, we have that $%
\omega (s,w)=f_{s}(w)\oplus iA,$ where $A$ is the union of $\infty $ and a
closed infinite interval of real numbers different from $\R$.

$(IV.4)$ For all $s\geq 0$ and for all $w\in \mathbb{H}$, we have that $%
\omega (s,w)=f_{s}(w)\oplus iA,$ where $A$ is a a closed finite interval of
real numbers. Moreover, the quantity
\begin{equation*}
\Theta :=\Re f_{s}(w)\ell _{\mathbb{H}}(\omega (s,w))\in (0,+\infty )
\end{equation*}%
does not depend neither on $s$ nor on $w.$

\smallskip

\noindent [Proof of Step IV] According to [Step III], there exists a non-empty subset $%
A\subset \mathbb{R}\cup \{\infty \}$ (only depending on the whole evolution
family) such that, for all $w\in \mathbb{H}$ and for all $s\geq 0,$%
\begin{equation*}
\omega (s,w)=f_{s}(w)\oplus iA.
\end{equation*}
Moreover, the set $A$ can be: $(a)$ a real number, $(b)$ the whole $\mathbb{R%
}$ union with $\infty $, $(c)$ a proper unbounded closed interval union with $\infty $, $(d)$ a compact interval. Let us analyze separately these four cases.

If case $(a)$ holds, then $A=\{a\}$ and, in particular, $\lim_{t\rightarrow
+\infty }\Im\phi _{s,t}(1)=a,$ for all $s\geq 0.$ Applying again [Step
II], we also obtain that, for every $s\geq 0,$ $\lim_{t\rightarrow +\infty
}\phi _{s,t}=\theta _{s}+ia$ uniformly on compact subsets of $\mathbb{H}$.
Of course, this case is in correspondence with $(IV.1).$ We leave the reader
to check that case $(b)$ is in correspondence with $(IV.2)$, case $(c)$ is
in correspondence with $(IV.3)$ and case $(d)$ is in correspondence with $%
(IV.4)$. The only point that really remains to prove is the assertion about $%
\Theta $ in $(IV.4).$ In this case, there exists two real numbers $\alpha
<\beta $ such that, for all $w\in \mathbb{H}$ and for all $s\geq 0,$%
\begin{equation*}
\omega (s,w)=\{f_{s}(w)+it:t\in \lbrack \alpha ,\beta ]\}.
\end{equation*}%
Let us calculate the hyperbolic length in $\mathbb{H}$ of this segment.
Namely,
\begin{equation*}
\ell _{\mathbb{H}}(\omega (s,w))=\int_{\omega (s,w)}\frac{|du|}{\Re u}%
=\int_{\alpha }^{\beta }\frac{dt}{\Re(f_{s}(w)+it)}=\frac{\beta
-\alpha }{\Re f_{s}(w)}.
\end{equation*}%
Clearly, the number $\Theta =\beta -\alpha $ does not depend neither on $s$ nor
on $w.$

\smallskip

[Proof of the Theorem] The theorem follows directly by translating [Step I]
and [Step IV] to the context of the unit disk by means of the Cayley map $%
\sigma _{\tau }$. At this respect, we recall that, for every $k>0,$
\begin{equation*}
\sigma _{\tau }\left( \mathrm{Hor}(\tau ,k)\right) =\{w\in \mathbb{H}: \Re w>\frac{1}{k}\}.
\end{equation*}%
We also recall that if $\gamma $ is an (absolutely continuous) curve in $\mathbb{D},$ then $\ell _{%
\mathbb{D}}(\gamma )=\ell _{\mathbb{H}}(\sigma _{\tau }\circ \gamma ).$
\end{proof}

\begin{rem}\label{remark Step II} Looking carefully at [Step II] of the above theorem, we see that the assertion $(ii)$ presented there is equivalent to the following fact: for every $w\in \mathbb{H}$ and every $s\geq 0,$ (resp. for some $w\in \mathbb{H}$ and some $s\geq 0)$
\begin{equation*}
\sup{\{\Re\phi _{s,t}(w):s\leq t}\}<+\infty.
\end{equation*}
\par
Moreover (maintaining the notations of the theorem) if, for some $w\in\h$, we know that $F(w,\cdot)$ belongs to $L^1([0,+\infty);\R)$, then statement $(ii)$ holds. We proceed arguing by contradiction. Namely, assume that $(i)$ holds and fix $w\in\h$. Now, since $F$ is a vector field associated with $(\phi_{s,t})$, some computations (see \cite[Theorem 7.1, Claim 2]{BCM1} for instance) show that, for every $t\geq 0$ and every $w\in\h$,
$$
\Re \phi_{0,t}(w)=\Re w\exp{\int_0^t \frac{\Re F(\phi_{0,\xi}(w),\xi)}{\Re \phi_{0,\xi}(w)}d\xi}.
$$
By hypothesis and bearing in mind the positivity property of $F$, we deduce that
$$
\lim_{t\to+\infty}\int_0^t \frac{\Re F(\phi_{0,\xi}(w),\xi)}{\Re \phi_{0,\xi}(w)}d\xi=+\infty.
$$
We know, by applying the Julia-Caratheodory Theorem, that $\Re \phi_{s,t}(w)\geq \Re w$. Hence, using \cite{Pomm79}, we have that, for every $\xi\geq 0$,
$$
\frac{\Re F(\phi_{0,\xi}(w),\xi)}{\Re \phi_{0,\xi}(w)}\leq \frac{\Re F(w,\xi)}{\Re w}.
$$
Therefore, the improper integral of the non-negative map $F(w,\cdot)$ in $[0,+\infty)$ is not (absolutely) convergent thus $F(w,\cdot)$ does not belong to $L^1([0,+\infty);\R)$.
\end{rem}

The five types of dynamical behavior mentioned in the above theorem are
indeed possible. In fact, they can be realized using what one might consider
the ``simplest'' cases of evolution families. Namely, the forthcoming examples
are evolution families in the unit disk formed by linear fractional maps and having
the point one as the common Denjoy-Wolff point. Moreover, they present a very
similar structure. That is, only little changes are needed to obtain all the
possible phenomena. The list goes as follows ($z\in \mathbb{D}$ and $0\leq s\leq t)$.

\begin{enumerate}
\item Example of convergence to the common Denjoy-Wolff point [Case (1)]:
\begin{equation*}
\varphi _{s,t}(z):=1+\frac{z-1}{1-(z-1)\left( -s+t+i(\sin (t)-\sin
(s)\right) }.
\end{equation*}

\item Example of convergence to a certain univalent function [Case (2)]:
\begin{equation*}
\varphi _{s,t}(z):=1+\frac{z-1}{1-(z-1)\left( e^{-s}-e^{-t}+i(\frac{1}{1+s}-
\frac{1}{1+t})\right) }.
\end{equation*}

\item Example where $\omega $-limits are complete circumferences [Case (3.a)]:
\begin{equation*}
\varphi _{s,t}(z):=1+\frac{z-1}{1-(z-1)\left( e^{-s}-e^{-t}+i(t\sin
(t)-s\sin (s))\right) }.
\end{equation*}

\item Example where $\omega $-limits are proper arcs having the common
Denjoy-Wolff point as one of its extremes [Case (3.b)]:
\begin{equation*}
\varphi _{s,t}(z):=1+\frac{z-1}{1-(z-1)\left( e^{-s}-e^{-t}+i(t^{2}\sin
(t)-s^{2}\sin (s))\right) }.
\end{equation*}

\item Example where $\omega $-limits are proper arcs inside the unit disk [Case (3.c)]:
\begin{equation*}
\varphi _{s,t}(z):=1+\frac{z-1}{1-(z-1)\left( e^{-s}-e^{-t}+i(\sin (t)-\sin
(s))\right) }.
\end{equation*}

\end{enumerate}

Now, we treat what we have called the inner case.

\begin{thm}\label{mainth inner}
Let $(\varphi _{s,t})$ be a non-trivial evolution family in the unit disk with common
Denjoy-Wolff point $\tau \in \mathbb{D}$. Then, one and only one of the
three mutually excluding situations can happen:

\begin{enumerate}
\item For every $s\geq 0,$ we have
$$
\lim_{t\longrightarrow +\infty }\varphi
_{s,t}=\tau
$$
uniformly on compact subsets of $\mathbb{D}$.

\item For every $s\geq 0,$ there exists a univalent holomorphic self-map of the unit disk
$h_{s}$ such that
$$
\lim_{t\longrightarrow +\infty }\varphi _{s,t}=h_{s}
$$
uniformly on compact subsets of $\mathbb{D}.$

\item For every $s\geq 0$ and
every $z\in \mathbb{D}\setminus\{\tau\},$ the $\omega $-limit $\omega (s,z)$ of the trajectory%
\begin{equation*}
t\in \lbrack s,+\infty )\mapsto \varphi _{s,t}(z)\in \mathbb{D}
\end{equation*}%
is a closed arc of the circumference defined by the boundary of a certain hyperbolic
disk $D_{H}(\tau ,r(s,z)),$ where $0<r(s,z)\leq\rho _{\mathbb{D}}(\tau,z)$.

Moreover, this case holds if and only if one of the following two
mutually excluding subcases holds:

\begin{enumerate}
\item For every $s\geq 0$ and for every $z\in \mathbb{D}\setminus\{\tau\}$, $\omega (s,z)$ is
exactly the whole circumference $\partial D_{H}(\tau ,r(s,z)).$

\item For every $s\geq 0$ and for every $z\in \mathbb{D}\setminus\{\tau\}$, $\omega (s,z)$ is
a proper closed arc of $\partial D_{H}(\tau ,r(s,z))$ and all of those arcs have the same
associated hyperbolic angular extent.
\end{enumerate}
\end{enumerate}
\end{thm}

\begin{proof}
First of all, we conjugate the evolution family to move the point $\tau$ to zero. That is, we
define
\begin{equation*}
\psi _{s,t}(w):=\alpha _{\tau }\circ \varphi _{s,t}\circ \alpha _{\tau }(z),%
\text{ }z\in \mathbb{D}\text{, }0\leq s\leq t,
\end{equation*}%
where $\alpha _{\tau }$ is the canonical map associated
with $\tau $, that is, $\alpha _{\tau }(z)=\dfrac{\tau -z}{1-\overline{\tau }%
z},$ $z\in \mathbb{D}.$ It is easy to see that $(\psi _{s,t})$ is also an
evolution family in the unit disk having zero as the common Denjoy-Wolff point.
Therefore, by Theorem \ref{BPorta formula with DW-common}, we know that an associated vector field is given by $F(z,t):=-zp(z,t),$ $z\in \mathbb{D},$ $t\geq 0,$
where $p:\mathbb{D}\times \lbrack 0,+\infty )\rightarrow \mathbb{C}$ is a
certain generalized Herglotz function in the unit disk.

We begin by proving two steps of independent interest and, at the end, we
will properly present the proof of the theorem.

\smallskip

\noindent [Step I] There exists an evolution family $(\phi _{s,t})$ in $\mathbb{H}$
with common Denjoy-Wolff point $\infty$ such that
\begin{equation*}
\begin{array}{l}
(i)\text{ }\psi _{s,t}(e^{-w})=\exp \left(-\phi _{s,t}(w)\right) ,\text{ }%
0\leq s\leq t,\text{ }w\in \mathbb{H}\text{.} \\
(ii)\text{ }\phi _{s,t}(w+2\pi i)=\phi _{s,t}(w)+2\pi i,\text{ }w\in \mathbb{H}\text{.} \\
(iii)\text{ For every }0\leq s\leq t,\text{ the function }\phi _{s,t}\text{
is parabolic and satisfies that }\\
\qquad\qquad\qquad\lim_{w\rightarrow \infty }(\phi _{s,t}(w)-w)=-\mathrm{log}(\varphi^\prime_{s,t}(\tau))\\
\text{in the angular sense.}
\end{array}
\end{equation*}

\smallskip

\noindent [Proof of Step I] Since $p$ is a generalized Herglotz function in the unit
disk, we see that
\begin{equation*}
q:\mathbb{H}\times \lbrack 0,+\infty )\rightarrow \mathbb{C}\quad
(w,t)\mapsto q(w,t):=p(e^{-w},t)
\end{equation*}%
is also a generalized Herglotz function but now in the right-half plane. Moreover,
by Theorem \ref{Herglotz in H}, $q$ is an associated vector field of a
certain (indeed unique) evolution family in $\mathbb{H}$ with common
Denjoy-Wolff point $\infty$. Let us denote
such family as $(\phi_{s,t}).$ Then, if we fix $s\geq 0$
and $w\in \mathbb{H},$ the mapping $t\in \lbrack s,+\infty )\rightarrow \phi
_{s,t}(w)\in\h$ is differentiable almost everywhere and
\begin{equation*}
\frac{\partial}{\partial t}\phi_{s,t}(w)=q(\phi_{s,t}(w),t),\quad a.e.
\text{ in }t\in \lbrack s,+\infty ).
\end{equation*}
Moreover, also $a.e.$ in $t\in \lbrack s,+\infty )$, we have that
\begin{eqnarray*}
\frac{\partial}{\partial t}\left( e^{-\phi _{s,t}(w)}\right) &=&
-e^{-\phi_{s,t}(w)}\frac{\partial}{\partial t}\phi _{s,t}(w)=-e^{-\phi
_{s,t}(w)}q(\phi _{s,t}(w),t) \\
&=&-e^{-\phi _{s,t}(w)}p(e^{-\phi _{s,t}(w)},t).
\end{eqnarray*}
In other words, the map $t\in \lbrack s,+\infty )\rightarrow e^{-\phi
_{s,t}(w)}$ is the solution (in the weak sense) of the Cauchy problem
\begin{equation*}
\left\{
\begin{array}{ll}
\dot{u}=-up(u,t), & t\geq s \\
u(s)=e^{-w}. &
\end{array}%
\right.
\end{equation*}%
Hence, by the theorem of uniqueness of solutions, we deduce that
\begin{equation*}
\psi _{s,t}(e^{-w})=e^{-\phi _{s,t}(w)}.
\end{equation*}%
This proves statement $(i).$

Now, using $(i)$, for every $0\leq s\leq t$ and every $w\in \mathbb{H}$,
\begin{equation*}
e^{-\phi _{s,t}(w+2\pi i)}=\psi _{s,t}(e^{-w-2\pi i})=\psi
_{s,t}(e^{-w})=e^{-\phi _{s,t}(w)}.
\end{equation*}
Therefore, there exists $k_{s,t}(w)\in \mathbb{Z}$ such that $\phi
_{s,t}(w+2\pi i)=\phi _{s,t}(w)+2k_{s,t}(w)\pi i.$ Moreover, the continuity properties in $s$, $t$ and $w$ force 
the existence of some $K\in \mathbb{Z}$ such $k_{s,t}(w)=K,\,$\ for all $0\leq s\leq t $
and for all $w\in \mathbb{H}.$

On the other hand, for $0\leq s\leq t$,
\begin{eqnarray*}
\phi _{0,t}(1)+2K\pi i &=&\phi _{0,t}(1+2\pi i)=\phi _{s,t}(\phi
_{0,s}(1+2\pi i)) \\
&=&\phi _{s,t}(\phi _{0,s}(1)+2K\pi i)=\phi _{s,t}(\phi _{0,s}(1))+2KK\pi i
\\
&=&\phi _{0,t}(1)+2K^{2}\pi i.
\end{eqnarray*}
Hence, $K=0,1.$ Since all the functions $\phi _{s,t}$ are univalent, we see that $K=1$ and this proves statement $(ii).$

Fix $0\leq s\leq t.$ Then, using again $(i)$, for any $x>0,$ we have that
\begin{equation*}
\frac{\psi _{s,t}(e^{-x})}{e^{-x}}=\exp (-\phi _{s,t}(x)+x).
\end{equation*}%
Taking limits when $x$ tending to $+\infty$ and recalling that $\psi
_{s,t}^{\prime }(0)\neq 0,$ we obtain
\begin{equation*}
\lim_{x\rightarrow +\infty }(\phi _{s,t}(x)-x)=-\log (\psi _{s,t}^{\prime
}(0))=-\log (\varphi _{s,t}^{\prime
}(\tau)).
\end{equation*}%
In particular, $\lim_{x\rightarrow +\infty }\dfrac{\phi _{s,t}(x)}{x}=1$. That is, the angular derivative of
$\phi_{s,t}$ at $\infty$ is one so $\phi _{s,t}$ is parabolic. Hence, by
the Julia-Caratheodory Theorem, either $\phi _{s,t}-\mathrm{id}_\h$ is constant or belongs to $\mathrm{Hol}(\mathbb{H};\mathbb{H})$. Finally, applying Lindeloff's Theorem, we deduce that
$\lim_{w\rightarrow \infty }(\phi _{s,t}(w)-w)=-\log (\varphi _{s,t}^{\prime
}(\tau))$
in the angular sense.

\smallskip

\noindent [Step II] One and only one of the following four cases holds:

$(II.1)$ For every $s\geq 0,$ $\lim_{t\rightarrow +\infty }\psi _{s,t}=0$
uniformly on compact subsets of $\mathbb{D}.$

$(II.2)$ For all $s\geq 0,$ there exists a univalent self-map of the unit
disk $h_{s}$ such that $\lim_{t\rightarrow +\infty }\psi _{s,t}=h_{s},$
uniformly on compact subsets of $\mathbb{D}$.

$(II.3)$ For all $s\geq 0$ and for all $z\in \mathbb{D\setminus}\{0\},$
there exists $0<r(s,z)\leq |z|$ such that $\omega (s,z)=C(0,r(s,z)).$

$(II.4)$ For all $s\geq 0$ and for all $z\in \mathbb{D\setminus}\{0\}$,
there exists $0<r(s,z)\leq |z|$ such that $\omega (s,z)$ is a proper
closed arc of the circumference $C(0,r(s,z)).$ Moreover, the euclidian
angular extent of $\omega (s,z)$ does not depend neither on $s$ nor on $z.$

\smallskip

\noindent [Proof of Step II] According to [Step I], there exists an evolution family
$(\phi _{s,t})$ in $\mathbb{H}$ with common Denjoy-Wolff point $\infty$ such that
\begin{equation*}
\psi _{s,t}(e^{-w})=\exp \left( -\phi _{s,t}(w)\right) ,\text{ }%
0\leq s\leq t,\text{ }w\in \mathbb{H}\text{.}
\end{equation*}

Moreover, applying [Step II] of Theorem \ref{mainth boundary} to $(\phi _{s,t})$ (see also Remark \ref{remark Step II}), we find that
one and only one of the following two cases holds:

$(i)$ For every $w\in \mathbb{H}$ and for every $s\geq 0,$ $%
\lim_{t\rightarrow +\infty }\Re\phi _{s,t}(w)=+\infty $.

$(ii)$ For every $w\in \mathbb{H}$ and for every $s\geq 0,$
$$
\beta_{s}(w):=\lim_{t\rightarrow +\infty }\Re\phi _{s,t}(w)<+\infty
$$ and $\beta _{s}$ is a harmonic function of $\mathbb{H}.$

Assume that we are in case $(i).$ Fix $s\geq 0$ and $z\in \mathbb{D\setminus
\{}0\mathbb{\}}$ and take any $w\in \mathbb{H}$ such that $e^{-w}=z.$ Then,
\begin{equation*}
|\psi _{s,t}(z)|=|\psi _{s,t}(e^{-w})|=|\exp \left( -\phi _{s,t}(w)\right)
|=\exp \left( -\Re\phi _{s,t}(w)\right) \overset{t\rightarrow +\infty }%
{\longrightarrow }0.
\end{equation*}%
Then, applying Vitali's Theorem, we deduce that $\lim_{t\rightarrow +\infty
}\psi _{s,t}=0$ uniformly on compact subsets of $\mathbb{D}.$ Clearly, this
corresponds to $(II.1).$

Assume now that we are in case $(ii).$ Applying [Step I] of Theorem \ref{mainth boundary} to
$(\phi _{s,t})$, we find that only two possible subcases can arise:

$(ii.a)$ For every $w\in \mathbb{H}$ and every $s\geq 0,$ $%
\lim_{t\rightarrow +\infty }\Im\phi _{s,t}(w)=\infty .$

$(ii.b)$ For every $w\in \mathbb{H}$ and every $s\geq 0,$ the family
$(\phi _{s,t}(w))_{t}$ does not converges to $\infty.$

If case $(ii.a)$ holds, bearing in mind the continuity of $(\phi _{s,t})$
with respect to $t,$ we find that, for every $s\geq 0$ and $z\in \mathbb{D\setminus}\{0\},$
\begin{equation*}
\omega (s,z)=C(0,e^{-\beta _{s}(w)}),\text{ }
\end{equation*}
where $w\in \mathbb{H}$ is any number such that $e^{-w}=z$. Using the Schwarz
Lemma, we note that
\begin{eqnarray*}
0<\exp (-\beta _{s}(w)) &=&\lim_{t\rightarrow +\infty }\exp (-\Re\phi
_{s,t}(w))=e^{-\Re w}\lim_{t\rightarrow +\infty }|\exp (-\phi
_{s,t}(w)+w)| \\
&=&e^{-\Re w}\lim_{t\rightarrow +\infty }|\frac{\psi _{s,t}(z)}{z}|\leq
e^{-\Re w}=|z|.
\end{eqnarray*}%
This subcase corresponds to $(II.3).$

Let us pass to case $(ii.b).$ Now, we are assuming $(\phi _{s,t})$ fails [Step I] from Theorem \ref{mainth boundary}. Therefore, we can apply [Step IV] from Theorem \ref{mainth boundary} to that family $(\phi _{s,t}).$ Hence, the
following three subcases can happen:

$(ii.b.1)$ For all $s\geq 0,$ there exists a univalent function $f_{s}\in
\mathrm{Hol}(\mathbb{H};\mathbb{H})$ such that $\lim_{t\rightarrow +\infty
}\phi _{s,t}=f_{s}$ uniformly on compact subsets of $\mathbb{H}$.

$(ii.b.2)$ For all $s\geq 0,$ there exists $f_{s}\in \mathrm{Hol}(\mathbb{H};
\mathbb{H})$ such that, for all $w\in \mathbb{H}$, we have that the $\omega$-limit of the trajectory $t\mapsto\phi _{s,t}(w)$
is $f_{s}(w)\oplus iA,$ where $A$ is the union of $\infty $ and a closed unbounded interval of real numbers (it could be proper or the whole real line). Moreover, $\Re f_{s}(w)\geq \Re w,$ for each $w\in \mathbb{H}.$

$(ii.b.3)$ For all $s\geq 0,$ there exists $f_{s}\in \mathrm{Hol}(\mathbb{H};%
\mathbb{H})$ such that, for all $w\in \mathbb{H}$, we have that the $\omega$-limit of the trajectory $t\mapsto\phi _{s,t}(w)$
is $f_{s}(w)\oplus iA,$ where $A$ is a closed finite interval of real
numbers. Moreover, $\Re f_{s}(w)\geq \Re w,$ for each $w\in \mathbb{H}.$

We begin by analyzing case $(ii.b.1)$. Fix $s\geq 0$ and $z\in \mathbb{%
D\setminus \{}0\mathbb{\}}$ and take any $w\in \mathbb{H}$ such that $%
e^{-w}=z.$ Then,
\begin{equation*}
\psi _{s,t}(z)=\psi _{s,t}(e^{-w})=\exp \left( -\phi _{s,t}(w)\right)
\overset{t\rightarrow +\infty }{\longrightarrow }\exp \left(
-f_{s}(w)\right) .
\end{equation*}%
Trivially, $\psi _{s,t}(0)=0,$ so the family $(\psi _{s,t})_t$ is pointwise
convergent in the unit disk. At the same time, that family is also a
relatively compact subset of $\mathrm{Hol}(\mathbb{D};\mathbb{C})$. Hence,
there exists a holomorphic map $h_{s}\in\mathrm{Hol}(\mathbb{D};\mathbb{C})$ such that
$\lim_{t\rightarrow +\infty }\psi _{s,t}=h_{s}$ uniformly on compact subsets
of $\mathbb{D}$. Obviously, $h_{s}(z)=\exp \left( -f_{s}(w)\right) ,$ for any
couple $(z,w)\in \mathbb{D}\setminus \{0\}\times \mathbb{H}$ such that $%
e^{-w}=z.$ \ Let us show that $h_{s}$ is univalent. Since every $\psi _{s,t}$
is univalent, according to Hurwitz's Theorem, either $h_{s}$ is univalent or is a constant.
In this last case, that constant should be zero so we would have that
$\exp \left( -f_{s}(w)\right) =0$ which is absurd.
Of course, this case corresponds to $(II.2).$

Now, we treat cases $(ii.b.2)$ and $(ii.b.3)$ together. Fix $s\geq 0$ and $%
z\in \mathbb{D\setminus \{}0\mathbb{\}}$ and take any $w\in \mathbb{H}$ such
that $e^{-w}=z$ Then, computing we find
\begin{equation*}
\omega (s,z)=\{e^{-\Re f_{s}(w)}e^{-i\Im f_{s}(w)}e^{-it}:t\in A\}.
\end{equation*}%
If the length of the closed interval $A$ is greater or equal than $2\pi$ (this happens always in case $(ii.b.2)$),
we deduce that $\omega (s,z)$ is the whole circumference $C(0,e^{-\Re f_{s}(w)}).$ This corresponds to $(II.3)$. Otherwise, $\omega (s,z)$ is
a proper closed arc of that circumference $C(0,e^{-\Re f_{s}(w)}).$ In this case, the euclidian angular extent of the above arc only depends on $A$ which is independent on $s$ and on $w$. Certainly, this case is in correspondence with $(II.4).$ We also note that, in both cases and if $e^{-w}=z$,
\begin{equation*}
e^{-\Re f_{s}(w)}\leq e^{-\Re w}=|e^{-w}|=|z|.
\end{equation*}

\smallskip

\noindent [Proof of the Theorem] By means of the canonical map $\alpha_\tau$, the theorem follows directly by translating to the family $(\varphi_{s,t})$ the results given in [Step II] for the family $(\psi_{s,t})$. At this respect, we recall that, for every $r\in(0,1),$
\begin{equation*}
\alpha_{\tau }(C(0,r))=C_H(\tau,\mathrm{log}\frac{1+r}{1-r}).
\end{equation*}
\end{proof}

\begin{rem} The statement given in [Step I] tell us that each evolution family in the unit disk with a common Denjoy-Wolff point in $\D$ is deeply related to another evolution family in the right half-plain having $\infty$ as the common Denjoy-Wolff point with all of their iterates having finite angular derivative of second order at $\infty$ (see \cite{Contreras-Diaz-Pommerenke:Derivada-segunda} for more details).
\end{rem}

\begin{rem} In classical Loewner theory as well as in the semigroup framework, only cases $(1)$ and $(3.a)$ appear. Indeed, in this last case, the item $(1)$ is related to the so-called Continuous Denjoy-Wolff Convergence Theorem and the second item is linked to the dynamical behavior of semigroups of elliptic automorphisms (see also \cite{Shoiket}).
\end{rem}

It is also possible to provide easy examples of the four types of dynamical behavior mentioned in the above theorem. Indeed, in this inner context, we can provide them using evolution families whose elements are just linear functions. Namely, consider
$$
\varphi_{s,t}(z):=\exp{(\nu(s)-\nu(t)})z,\quad z\in\D,\quad 0\leq s\leq t.
$$
Then, choosing
\begin{enumerate}
\item $\nu(t)=t$, we have an example of case (1).
\item $\nu(t)=\arctan(t)$, we have an example of case (2).
\item $\nu(t)=it$, we have an example of case (3.a).
\item $\nu(t)=i\pi|\sin(t)|$, we have an example of case (3.b).
\end{enumerate}


\section{Additional Results}

\subsection{The Role of the Spectral Function}
From the very beginning, derivatives of the maps $\varphi_{s,t}$ at the fixed point zero have played a prominent role in asymptotic questions (also in Loewner parametrization method) for the Radial Loewner Equation. The idea behind is a normalization procedure which seems to go back to Koenings linearization method \cite[Theorem 8.2]{Milnor}. There, it is important (and easy to check) that $\varphi_{s,t}^\prime(0)=\exp{(\lambda(s)-\lambda(t))}$ for a certain map $\lambda$ called the spectral function of the family. In \cite{BCM1}, the following general result about the existence of such functions is given.

\begin{thm}\label{existencia de lambda}
Let $(\varphi _{s,t})$ be a non-trivial evolution family in $\mathbb{D}$ with
common Denjoy-Wolff point $\tau \in \overline{\mathbb{D}}$. Then, there exists a
unique spectral function $\lambda~:[0,+\infty )\longrightarrow
\mathbb{C}$ such that
\begin{equation}
\varphi _{s,t}^{\prime }(\tau )=\exp (\lambda (s)-\lambda (t)),\text{ }0\leq
s\leq t.
\end{equation}%
Moreover, if $\tau \in \partial \mathbb{D}$, then $\lambda (t)\in \lbrack
0,+\infty ),$ for every $t\geq 0.$

As usual, by spectral function, we mean that
\begin{enumerate}
\item $\lambda$ is locally absolutely continuous.
\item $\lambda (0)=0.$
\item For any $t\geq 0,$ $\Re\lambda (t)\geq 0.$
\item The map $t\in \lbrack 0,+\infty )\longrightarrow \Re\lambda
(t)\in \mathbb{R}$ is non-decreasing.
\end{enumerate}
\end{thm}

\begin{rem}
When $\tau \in \partial \mathbb{D}$, the derivative in the above theorem
must be understood in the angular sense. It is easy to see that, given a
spectral function $\lambda $, there are many evolution families with common
Denjoy-Wolff point $\tau $ having $\lambda $ as the associated spectral
function. Moreover, $\tau $ can be located arbitrarily in
$\overline{\mathbb{D}}.$
\end{rem}
Let us analyze the role of $\lambda$ in our version of the Denjoy-Wolff Theorem.
\subsubsection{The boundary case}
We maintain the notations of Theorem \ref{mainth boundary} and its proof so $(\phi_{s,t})$ denotes the associated evolution family in the right half-plane. Since $\lambda $ is
non-decreasing and takes always non-negative real values, the
limit $L:~=\lim_{t\rightarrow +\infty }\lambda (t)\in \lbrack 0,+\infty ]$
always exists. If $L=+\infty $ and, applying Theorem \ref{existencia de lambda}, we find that
\begin{equation*}
\Re\phi _{0,t}(1)\geq e^{\lambda (t)}\Re w\overset{t\rightarrow
+\infty }{\longrightarrow }+\infty
\end{equation*}%
so $\lim_{t\rightarrow +\infty }\phi _{0,t}(1)=\infty .$ Hence, we are in case $(1)$ of Theorem \ref{mainth boundary}. The other possibility is that $L\in \lbrack 0,+\infty
).$ Then, for every $w\in \mathbb{H}$ and every $s\geq0$,
\begin{eqnarray*}
\Re f_{s}(w) &=&\Re\theta _{s}(w):=\lim_{t\rightarrow +\infty }\Re\phi
_{s,t}(w)\\&\geq& \lim_{t\rightarrow +\infty }e^{\lambda (t)-\lambda (s)}\Re w=e^{L-\lambda (s)}\Re w.
\end{eqnarray*}
Notice that now we are in cases $(2)$ or $(3)$ and, in this last case, the value $L$ is related to the factor of the associated horocycle mentioned in that theorem. Certainly, the type of convergence showed in case $(1)$ can even happen when $L$ is finite.

\subsubsection{The inner case}
We also maintain the notations of Theorem \ref{mainth inner} and its proof so $(\psi_{s,t})$ denotes the associated evolution family in $\D$ with zero as the common Denjoy-Wolff point. Recall that $\varphi_{s,t}^\prime(\tau)=\psi_{s,t}^\prime(0)$, for all $0\leq s\leq t$.

\begin{prop} As in former subsection, let $L:=\lim_{t\rightarrow +\infty }\Re\lambda (t)\in \lbrack 0,+\infty ]$. Then, $L=+\infty$ if and only if for any (resp. some) $s\geq0$, the family $(\psi_{s,t})_t$ tends to zero uniformly on compact subsets of $\D$ when $t$ goes to $+\infty$.
\end{prop}
\begin{proof}
If $L=+\infty$, then the convergence to zero follows by applying directly the Distortion Theorem \cite[Theorem 1.6]{Pommerenke} and Theorem \ref{existencia de lambda}. On the other hand, if $(\psi_{s,t})_t$ tends to zero uniformly on compact subsets of $\D$ for some $s\geq 0$ then, by Weierstrass's Theorem, we also have that $(\psi_{s,t}^\prime)_t$ tends to zero. Therefore, applying again Theorem \ref{existencia de lambda}, we have $L=+\infty$.
\end{proof}

In case $L<+\infty$, both the real and the imaginary part of $\lambda$ play a significative role. As above, using the Distortion Theorem, we find $\Re\lambda$ partially controls the radius of the hyperbolic disk mentioned in Theorem \ref{mainth inner}.
\begin{prop} Assume that $L<+\infty$. Then, for each $z\in\D\setminus\{0\}$ and $s\geq0$, there  exists a real constant $c=c(s,z)$ such that $\omega_A=e^{ic}\omega_\lambda$, where $\omega_A$ is the $\omega$~-limit of the function $t\in[s,+\infty)\mapsto \exp{(i\mathrm{Arg}(\psi_{s,t}(z)}))$ and $\omega_\lambda$ is the $\omega$-limit of $t\in[0,+\infty)\mapsto \exp{(-i\lambda(t))}$.
\end{prop}
\begin{proof}
Fix $z\in\D\setminus\{0\}$ and $s\geq0$. Some computations involving differential equations (for instance, look at \cite[Section 7]{BCM1}) show that, for all $t\geq s$,
$$
\psi_{s,t}(z)=z\exp{\left(-\int_s^t p(\psi_{s,\xi}(z),\xi)d\xi\right)}
$$
where $F(z,t):=-zp(z,t)$ is an associated vector field and $p$ is the corresponding generalized Herglotz function in the unit disk (see Theorem \ref{BPorta formula with DW-common}). By \cite[Section 7.1, Exercise 2]{Conway} and Schwarz's lemma, we obtain that, for every $\xi\geq s$,
$$
\left\vert\Im p(0,\xi)-\Im p(\psi_{s,\xi}(z) \right\vert\leq 2\Re p(0,\xi)\frac{|z|}{1-|z|^2}.
$$
Since $\Re\lambda(t)$ has finite limit when $t$ tends to $\infty$ and bearing in mind (see again \cite[Section 7]{BCM1})
$$
\lambda(t)=\int_0^t p(0,\xi)d\xi, \text{ for all } t\geq0,
$$ we conclude the integral
$$
\int_s^\infty (\Im p(0,\xi)- \Im p(\psi_{s,\xi}(z),\xi))d\xi
$$
is absolutely convergent with value $c=c(s,z)\in\mathbb{R}$. Now, for $t\geq s$, consider the identity
$$
e^{i\mathrm{Arg}(\psi_{s,t}(z))}=e^{-i\Im\lambda(t)}e^{i\mathrm{Arg}(z)}e^{i\Im\lambda(s)}
\exp{\left(i\int_s^t (\Im p(0,\xi)- \Im p(\psi_{s,\xi}(z),\xi))d\xi\right)}.
$$
Hence, we deduce that $\omega_A=e^{i(\mathrm{Arg (z)}+\Im\lambda (s)+c)}\omega_\lambda$.
\end{proof}

\subsection{Automorphisms}
The examples given in Section 3 suggest to analyze if there is any special feature in the statement of our two main theorems when dealing with evolution families formed by automorphisms. First of all, we want to clarify how automorphisms can appear in those families since it is no longer true (like in the semigroup context or discrete iteration) that if one $\varphi_{s,t}$ ($s<t$) is an automorphism so are the other iterates. In what follows, $\mathrm{Aut}(\D)$ will denote the set of all holomorphic automorphisms of the unit disk.

\begin{prop} Let $(\varphi_{s,t})$ an evolution family. Then, either, for every $s\geq0$,
$$
\sup\{t\in[s,+\infty):\varphi_{s,t}\in\mathrm{Aut}(\D)\}<+\infty,
$$
or there exists $\alpha\geq0$ such that:
\begin{enumerate}
\item For every $0\leq\alpha\leq s\leq t$, $\varphi_{s,t}\in\mathrm{Aut}(\D)$.
\item (if $\alpha>0$) There exists a family $(h_r)_{0\leq r<\alpha}$ of holomorphic self-maps of the unit disk with no $h_r$ belonging to $\mathrm{Aut}(\D)$ such that $\lim_{r\to\alpha}h_r=\mathrm{id}_\D$ locally uniformly in $\D$ and $\varphi_{r,t}=\varphi_{\alpha,t}\circ h_r$, for every $0\leq r<\alpha\leq t$.
\end{enumerate}
\end{prop}

It is easy to find examples where both situations in the above theorem hold exactly. It is quite natural to name those evolution families described in the second option as evolution families of automorphic type.
\begin{proof}
Assume there exists $s_0\geq0$ with
$$
\sup\{t\in[s_0,+\infty):\varphi_{s_0,t}\in\mathrm{Aut}(\D),
\text{ for all} t\geq s\}=+\infty.
$$
Now, using that all iterates are univalent as well as the algebraic composition property of evolution families we obtain that, indeed, $\varphi_{s,t}\in\mathrm{Aut}(\D)$, for every $s_0\leq s\leq t$. Now, define
$$
\alpha:=\inf\{s\geq0: \varphi_{s,t}\in\mathrm{Aut}(\D)\}.
$$
By the above comment and continuity properties of evolution families in the variables $s$ and $t$, we deduce that $\varphi_{s,t}\in\mathrm{Aut}(\D)$, for every $\alpha\leq s\leq t$. Moreover, whenever $\alpha>0$, for every $0\leq r\leq\alpha$,
$$
h_r:=\varphi_{r,\alpha}\notin \mathrm{Aut}(\D).
$$
It is straightforward to check the family $(h_r)$ has the other properties mentioned in statement $(2)$.
\end{proof}

We will use the following result which has some interest in its own. It might be considered a natural version (an almost everywhere version) for evolution families of a well-known result due to Berkson-Porta for the unit disk and later extended to complex manifolds (see \cite[Section 3.2]{Shoiket}).

\begin{thm}\label{BP ae-everywhere} Let $(\varphi_{s,t})$ an evolution family in the unit disk. Then, there exists $A\subset[0,+\infty)$ a subset of null measure such that, for every $z\in\D$ and every $t~\in(0,+\infty)\setminus A$, the following limit exists
$$
G(z,t):=\lim_{h\rightarrow0^{+}}=\frac{\varphi_{t,t+h}(z)-z}{h}.
$$
Moreover, $G$ defines (almost everywhere) a vector field associated with $(\varphi_{s,t}).$
\end{thm}

\begin{proof}According to \cite[Theorem 6.4]{BCM1}, there exists $A\subset[0,+\infty)$ of null measure such that, for all $t\in(0,+\infty)\setminus A$, we can assure, for all $z\in\D$, the existence of the limit
$$
\frac{\partial \varphi}{\partial t}(z,0,t)=\lim_{h\rightarrow0^{+}}\frac{\varphi_{0,t+h}(z)-\varphi_{0,t}(z)}{h}
$$
and, indeed, that limit defines a holomorphic function in $\D$.
Set $B_t:=\varphi_{0,t}(\D)$, for all $t\geq0$. Since $\varphi_{0,t}$ is univalent, we see that $B_t$ is a complex domain contained in the unit disk.

Now, fix $t\geq 0$ and $z\in B_t$ and take $\widetilde{z}\in\D$ such that $z=\varphi_{0,t}(\widetilde{z})$. Using the algebraic properties of evolution families, we deduce that the following limits also exist
$$
\frac{\partial \varphi}{\partial t}(\widetilde{z},0,t)=\lim_{h\rightarrow0^{+}}
\frac{\varphi_{0,t+h}(\widetilde{z})-\varphi_{0,t}(\widetilde{z})} {h}
=\lim_{h\rightarrow0^{+}}\frac{\varphi_{t,t+h}(z)-z}{h}.
$$
Moreover, by \cite[Remark 3.3.1 and Corollary 3.3.1]{Shoiket}, we know that, all the elements from
$$
\Gamma=\{\frac{1}{h}(\varphi_{t,t+h}-\mathrm{id}_\D):h>0\}
$$
are infinitesimal generators of semigroups of analytic functions. Bearing in mind \cite[Lemma 2.2]{BCM1} and the above convergence fact over the subset $B_t$, we conclude that $\Gamma$ is a relatively compact subset of $\mathrm{Hol}(\D,\mathbb{C})$. The proof ends by using again the just mentioned convergence property and the identity principle for analytic functions.
\end{proof}

\subsubsection{The boundary case}
Assume that $(\varphi_{s,t})$ denotes a non-trivial evolution family of automorphic type (we put $\alpha=0$ for simplicity) in $\D$ with $\tau$ as the common Denjoy-Wolff point and associated spectral function $\lambda$. Likewise, let $(\phi_{s,t})$ be the associated evolution family in the right-half plane with $\infty$ as the common Denjoy-Wolff point (see subsection \ref{sub DW-point}). Since each $\phi_{s,t}$ is again an automorphism and using the Julia-Caratheodory's Lemma, we obtain that, for every $w\in\h$ and $0\leq s\leq t$,
$$
\phi_{s,t}(w)=\exp{\left(\lambda(t)-\lambda(s)\right)}w+ib(s,t),
$$
where $b(s,t)\in\mathbb{R}$. By means of Theorem \ref{BP ae-everywhere} and integrating the resultant ODE, we know that there exists $h\in L^1([0,+\infty),\mathbb{R})$ with
$$
\varphi_{s,t}(w)=\exp{\left(\lambda(t)-\lambda(s)\right)}w
+ie^{\lambda(t)}\int_s^t h(\xi)e^{-\lambda(\xi)}d\xi.
$$
Coming back to Theorem \ref{mainth boundary} (items (2) and (3)) and bearing in mind the above expression, we see that in this case we can give the additional information that
$$
k(z,s)=e^{\Re\lambda(s)-L}k_\D(z,s),
$$
where $L=\lim_{t\rightarrow+\infty}\Re\lambda(t)$.
\subsubsection{The inner case}
Assume that $(\psi_{s,t})$ denotes an evolution family of automorphic type in $\D$ with zero as the common Denjoy-Wolff point and associated spectral function $\lambda$ (again we put $\alpha=0$ for simplicity).

In this case, standard arguments using the Schwarz Lemma show that $\lambda(t)=ib(t)$, with $b(t)\in\mathbb{R}$ and
$$
\psi_{s,t}(z)=\exp{(ib(s)-ib(t))}, \quad z\in\D,\quad 0\leq s\leq t.
$$
Therefore, in this context, we are always in case $(3.a)$ of Theorem \ref{mainth inner}.

\subsection{Non-Tangential Convergence}
When the Denjoy-Wolff point $\tau$ of a holomorphic self-map $\varphi$ of the unit disk belongs to $\partial\D$, one of the most treated associated dynamical problem is to analyze, for each $z\in\D$, whether the sequence $(\varphi_n(z))$ approaches $\tau$ non-tangentially, that is, if $(\varphi_n(z))$ is included in some Stolz angle linked at $\tau$. This is a problem not completely understood yet \cite{Contreras-Diaz-Pommerenke:zoo}. The corresponding one for semigroups has also attracted great attention (see \cite{Asintotico Shoikhet} and the references therein). In our setting, and following the line of what has been done for semigroups, the problem would be to study, for each $s\geq $ and each $z\in\D$, when $\{\varphi_{s,t}(z):s\leq t\}$ is included in some Stolz angle linked at $\tau$.

Our next result shows the above problem does not depend neither on $z$ nor in $s$.
\begin{thm} Let $(\varphi_{s,t})$ be a non-trivial evolution family in the unit disk with common Denjoy-Wolff point $\tau\in\partial\D$. Then, the following are equivalent:
\begin{enumerate}
\item There exist $z\in\D$ and $s\geq0$ such that $(\varphi_{s,t}(z))_t$ converges to $\tau$ non-tangentially when $t$ tends to $\infty$.
\item For every $z\in\D$ and every $s\geq0$, the family $(\varphi_{s,t}(z))_t$ converges to $\tau$ non-tangentially when $t$ tends to $\infty$.
\end{enumerate}
\end{thm}
\begin{proof}
Fix $z_0\in\D$ and $s_0\geq 0$ and assume that $(\varphi_{s_0,t}(z_0))_t$ converges to $\tau$ non-tangentially. By Theorem \ref{mainth boundary}, we know that, indeed, $(\varphi_{s,t}(z))_t$ converges to $\tau$ for every $s\geq 0$ and every $z\in\D$. Also we know that, for some $0<\alpha<\pi/2$,
$$
\{\varphi_{s_0,t}(z_0):s\leq t\}\subset S_\alpha:=\{z\in\D:|\mathrm{Arg}(1-\overline{\tau}z)|<\alpha\}.
$$
Now, consider $z\neq z_0$ with $z\in\D$. Then, using the Schwarz-Pick Lemma, we find that,
$$
\{\varphi_{s_0,t}(z):s\leq t\}\subset D_H(z_0,R), \text{ where } R:=\rho _{\mathbb{D}}^{{}}(z_0,z)>0.
$$
We claim that this implies the existence of a Stolz angle $S_\beta$ with $0<\alpha<\beta <\pi/2$ such that $\{\varphi_{s,t}(z):s\leq t\}\subset S_\beta.$ Moreover, bearing in mind that $\varphi_{s,t}=\varphi_{s_0,t}\circ\varphi_{s,s_0}$, for $0\leq s\leq s_0$, we deduce that $(\varphi_{s,t}(z))_t$ converges to $\tau$ non-tangentially, for every $z\in\D$ and every $0\leq s\leq s_0$.
Finally, consider $s\geq s_0$ and $z\in\D$. Since $\varphi_{s_0,t}(z)=\varphi_{s,t}(\varphi_{s_0,s}(z))$, we have a point $u\in\D$ such that $(\varphi_{s,t}(u))_t$ converges to $\tau$ non-tangentially. Then, repeating the argument given above, we also conclude that $(\varphi_{s,t}(z))_t$ converges to $\tau$ non-tangentially.

\noindent Proof of the claim. Consider, for every $t\geq 0$,
$$
c_t:=\sigma_\tau(\varphi_{s_0,t}(z_0))\in\h,\quad w_t=\sigma_\tau(\varphi_{s_0,t}(z))\in\h.
$$
By hypothesis, we know that $w_t$ belongs to the hyperbolic disk $D_H^\h(c_t,R)$ in $\h$ centered at $c_t$ and radius $R$. Moreover, also for all $t\geq0$,
$$
\left\vert\frac{\Im c_t}{\Re c_t}\right\vert\leq \tan(\alpha).
$$
Now, take the unique automorphism $A_t$ in $\h$ which fixes $\infty$ and maps $c_t$ into $1$. Some computations show that $A_t(D_H^\h(1,R))=D_H^\h(c_t,R)$. Therefore, there exists $u_t\in D_H^\h(1,R)$ such that $A_t(u_t)=w_t$. Trivially, $D_H^\h(1,R)$ is included in some Stolz angle in $\h$ linked at zero and $I:=\inf\{\Re w:w\in D_H^\h(1,R)\}<+\infty$. Therefore,
\begin{eqnarray*}
\sup\{\left\vert\frac{\Im w_t}{\Re w_t}\right\vert:t\geq0\}&=&\sup\{\left\vert
\frac{\Re c_t\Im u_t+\Im c_t}{\Re c_t\Re u_t}\right\vert:t\geq0\}\\
&\leq&I \sup\{\left\vert\frac{\Im u_t}{\Re u_t}\right\vert:t\geq0\}\sup\{\left\vert\frac{\Im c_t}{\Re c_t}\right\vert:t\geq0\}<+\infty.
\end{eqnarray*}
Thus, $\{w_t:t\geq 0\}$ is included in a certain Stolz angle in $\h$ linked at zero so $\{\varphi_{s,t}(z))_t:t\geq 0\}$ is included in a certain Stolz angle in $\D$ linked at $\tau$.
\end{proof}



\end{document}